\documentclass{amsart}
\usepackage{amssymb,amsmath,latexsym}
\usepackage{amsthm}
\usepackage{fontenc}
\usepackage{amssymb}
\numberwithin{equation}{section}

\newtheorem{theorem}{Theorem}[section]

\newtheorem{proposition}[theorem]{Proposition}

\begin{document}
\author{Alexander E Patkowski}
\title{On some instances of Fox's Integral Equation}

\maketitle
\begin{abstract}We consider some applications of the non-homogeneous second order integral equation of Fox. Some new
solutions to Fox's integral equation are discussed in relation to number theory.  \end{abstract}



\section{Introduction}
The Fredholm non-homogeneous second-order integral equation has the form
\begin{equation} \Delta(x)=f(x)+\int_{0}^{\infty}k(x,t)\Delta(t)dt,\end{equation}
where $0<x<\infty.$ Integral equations of this form are known to have many applications, particularly in boundary value problems. Fox [1] studied the special case $k(x,t)=k(xt),$ and noted that
a simple method to solve (1.1) would be to employ Mellin transforms. The general (formal) solution
is given by 
\begin{equation}\Delta(x)=\frac{1}{2\pi i}\int_{c-i\infty}^{c+i\infty}\frac{\bar{f}(s)+\bar{k}(s)\bar{f}(1-s)}{1-\bar{k}(s)\bar{k}(1-s)}x^{-s}ds,\end{equation}
where the bar signifies the Mellin transform of a suitable function $f$
\begin{equation}\bar{f}(s):=\int_{0}^{\infty}t^{s-1}f(t)dt.\end{equation} Here we must assume that each of the Mellin transforms in (1.2) exist. 
The kernels $k(x,t)=\dot{k}(xt):=\sin(xt),$ and $k(x,t)=\ddot{k}(xt):=\cos(xt),$ correspond to the theory of Fourier sine and cosine transforms respectively [4]. In
fact, the reader that is familiar with such transforms might easily recognize solutions to (1.1) from tables [2]. Indeed, the Fourier
sine transform of $(e^{2\pi t}-1)^{-1},$ 
is given by 
\begin{equation} \int_{0}^{\infty}\frac{\sin(wt)dt}{e^{2\pi t}-1}=\frac{1}{2}\left(\frac{1}{e^{w}-1}+\frac{1}{2}-\frac{1}{w}\right).\end{equation}
Hence, a formal solution to the non-homogeneous second order integral equation
\begin{equation} \Delta(x)=f(x)+\int_{0}^{\infty}\sin(xt)\Delta(2\pi t)dt,\end{equation}
where $$f(x)=\frac{1}{2}\left(\frac{1}{2}-\frac{1}{x}\right)$$ is 
\begin{equation}\Delta(x)=\frac{1}{2}\frac{1}{e^{x}-1}.\end{equation} Instances of this sort appear to be ubiquitous. Furthermore, 
solutions to the case $f(x)=0$ in (1.1), which gives the first-order homogeneous equation, are essentially Fourier pairs when we choose the kernels
$\dot{k}(xt),$ $\ddot{k}(xt).$ From the above observations we may make the following statement. Let $\zeta(s)=\sum_{n\ge1}n^{-s}$ denote the Riemann zeta function [6].

\begin{proposition} Put $\dot{\gamma}(x):=\frac{\bar{h}(1)}{x}-\frac{h(0+)}{2},$ where $\bar{g}$ is the Mellin transform of $g,$ and $g$ is bounded at the origin. If $\Delta(x)=\sum_{n\ge1}h(nx),$ then the non-homogeneous second-order integral equation
\begin{equation} \Delta(x)=\dot{\gamma}(x)+\int_{0}^{\infty}k(xt)\Delta(2\pi t)dt,\end{equation}
may be transformed into the functional equation
\begin{equation}\bar{h}(s)\zeta(s)=(2\pi)^{s-1}\bar{k}(s)\zeta(1-s)\bar{h}(1-s).\end{equation} Furthermore, if $k(x)=2\sin(x)$ then
\begin{equation}\bar{h}(s)\zeta(s)=2(2\pi)^{s-1}\Gamma(s)\sin(\frac{\pi}{2}s)\zeta(1-s)\bar{h}(1-s).\end{equation}
\end{proposition}
\begin{proof} If we assume $\dot{\gamma}(x)$ is as defined in the proposition, and the solution may be written as the 
series $\Delta(x)=\sum_{n\ge1}h(nx),$ then we may write (for $-\epsilon<c<0,$ $\epsilon>0$)
\begin{equation}\sum_{n\ge1}h(nx)-\frac{\bar{h}(1)}{x}+\frac{h(0+)}{2}=\frac{1}{2\pi i}\int_{c-i\infty}^{c+i\infty}\bar{h}(s)\zeta(s)x^{-s}ds.\end{equation}
by [4, pg.118, eq.(4.1.3)]. Taking Mellin transforms of both sides of (1.10) over the appropriate region gives us
\begin{equation}\int_{0}^{\infty}t^{s-1}\left(\sum_{n\ge1}h(nx)-\frac{\bar{h}(1)}{x}+\frac{h(0+)}{2}\right)dt=\bar{h}(s)\zeta(s),\end{equation}
for $-\epsilon<\Re(s)<0.$ Therefore, after rearranging (1.7), taking Mellin transforms and invoking the series for $\Delta(x),$ we may 
obtain the proposition. In the case of (1.9), we use the known Mellin transform 
\begin{equation}\int_{0}^{\infty}t^{s-1}\sin(xt)dt=\frac{\Gamma(s)\sin(\frac{\pi}{2}s)}{x^s},\end{equation}
where $-1< \Re(s)<1.$
\end{proof}

Note that putting $h(x)=e^{-x}$ in Proposition 1.1 implies (1.4) may be transformed into the functional equation for the Riemann zeta function [6]:
\begin{equation}\zeta(s)=2(2\pi)^{s-1}\Gamma(1-s)\sin(\frac{\pi}{2}s)\zeta(1-s).\end{equation} Appropriate modifications may be made to (1.7)
in the proposition to work with different kernels in the functional equation (1.8). This idea would of course be used in studying other Dirichlet series.

\section{A slightly modified Fox integral equation}
Here we consider our main object of study and re-write the Fox equation as 
\begin{equation} \pi\Delta(a x)=-f(x)+\int_{0}^{\infty}k(xt)\Delta(t)dt,\end{equation}
$a\in\mathbb{R},$ and put $k(x)=\sin(x).$ In this case taking Mellin transforms gives us
\begin{equation} \pi a^{-s}\bar{\Delta}(s)=-\bar{f}(s)+\Gamma(s)\sin(\frac{\pi}{2}s)\bar{\Delta}(1-s).\end{equation}
Replacing $s$ by $1-s$ and applying some standard computations toward the formal solution (1.2), we have that
\begin{equation} \bar{\Delta}(s)=\frac{-\pi\bar{f}(s)}{\pi^2a^{-s}-a^{1-s}\frac{\pi}{2}}-\frac{a^{1-s}}{\pi^2a^{-s}-a^{1-s}\frac{\pi}{2}}\bar{k}(s)\bar{f}(1-s).\end{equation}
Hence,
\begin{equation} \Delta(x)=\frac{-\pi f(x/a)}{\pi^2-a\frac{\pi}{2}}-\frac{a}{\pi^2-a\frac{\pi}{2}}\int_{0}^{\infty}\sin(xt)f(t)dt.\end{equation}
Inserting (2.4) into (2.1) would give

\begin{align}\pi\Delta(a x)
&=-f(x)+\int_{0}^{\infty}k(xt)\Delta(t)dt \\ 
&=-f(x)+\int_{0}^{\infty}\sin(xt)\left(\frac{-\pi f(t/a)}{\pi^2-a\frac{\pi}{2}}-\frac{a}{\pi^2-a\frac{\pi}{2}}\int_{0}^{\infty}\sin(yt)f(y)dy\right)dt\\
&=-f(x)-\frac{\pi}{\pi^2-a\frac{\pi}{2}}\int_{0}^{\infty}\sin(xt)f(t/a)dt-\frac{a}{\pi^2-a\frac{\pi}{2}}\int_{0}^{\infty}\int_{0}^{\infty}\sin(xt)\sin(yt)f(y)dydt\\
&=-f(x)-\frac{\pi}{\pi^2-a\frac{\pi}{2}}\int_{0}^{\infty}\sin(xt)f(t/a)dt-\frac{a\frac{\pi}{2}}{\pi^2-a\frac{\pi}{2}}f(x)\\
&=-f(x)-\frac{\pi a}{\pi^2-a\frac{\pi}{2}}\int_{0}^{\infty}\sin(axt)f(t)dt-\left(\frac{\pi^2}{\pi^2-a\frac{\pi}{2}}-1\right)f(x)\\
&=\frac{-\pi a}{\pi^2-a\frac{\pi}{2}}\int_{0}^{\infty}\sin(axt)f(t)dt-\frac{\pi^2}{\pi^2-a\frac{\pi}{2}}f(x).
\end{align}
Now replacing $x$ by $x/a,$ and dividing both sides by $\pi$ in (2.10) gives (2.4). Therefore (2.4) is a formal solution to (2.1). However, it should 
be noted that the solution has a singularity at $a=2\pi,$ and so a different treatment appears to be warranted in this case. As it turns out (2.1) when $a=2\pi$ has some relevance to number theory. From [7] we find the interesting expansion attributed to J$\ddot{o}$rg Waldvogel valid for $t>0$
\begin{equation} R(e^{-2\pi t})=\frac{1}{\pi}\sum_{n\ge1}\frac{(-1)^{n-1}t^{-2n-1}}{(2n+1)\zeta(2n+1)}+\frac{1}{2}\sum_{\rho}\frac{t^{-\rho}}{\rho \cos(\pi  \rho/2) \zeta'(\rho)},\end{equation} where $\mu(n)$ is the M$\ddot{o}$bius function. The sum over $\rho$ is taken over all the non-real zero's of $\zeta(s).$
Here $R(x)$ is Riemann's approximation to the prime counting function (which counts the number of prime numbers $\le x$), typically written in the form
$$R(x)=\sum_{n\ge1}\frac{\mu(n)}{n}L(x^{1/n}),$$ where $L(x)$ is the logarithmic integral. We write (for $x>0$)
\begin{align} \int_{0}^{\infty}\sin(xt)\sum_{n\ge1}\frac{\mu(n)}{n}e^{-t/n}dt
&= \sum_{n\ge1}\frac{\mu(n)}{n}\frac{x}{x^2+1/n^2}\\ 
&=\frac{1}{x}\sum_{n\ge1}\frac{\mu(n)}{n}\sum_{k\ge0}(xn)^{-2k}(-1)^k\\
&=\frac{1}{x}\sum_{n\ge0}\frac{(-1)^nx^{-2n}}{\zeta(2n+1)}\\
&=-x\frac{d}{dx}\left(\frac{\pi}{2}\sum_{\rho}\frac{x^{-\rho}}{\rho \cos(\pi  \rho/2) \zeta'(\rho)}-\pi R(e^{-2\pi x})\right).
\end{align}
Here we note for (2.14) that $1/\zeta(1)=0.$ We therefore have the following solution.
\begin{theorem} A solution to the second-order Fredholm integral equation (2.1) when $a=2\pi,$ and
\begin{equation}f(x)=\frac{\pi}{2}\sum_{\rho}\frac{x^{-\rho}}{\cos(\pi  \rho/2) \zeta'(\rho)},\end{equation}
is given by \begin{equation}\Delta(x)=\sum_{n\ge1}\frac{\mu(n)}{n}e^{-x/n}.\end{equation}
\end{theorem}

\section{Comments} 
The Fredholm alternative theorem [7, pg.23, Theorem 1.3.4] implies that, if the homogeneous integral equation
\begin{equation}\pi\Delta(2\pi x)=\int_{0}^{\infty}\sin(xt)\Delta(t)dt\end{equation}
has nontrivial solutions (i.e. $\Delta(x) \neq 0$), then (2.1) with $a=2\pi$ has infinitely many solutions or no solution at all. From Titchmarsh [5, pg.23, eq. (2.7.1)--(2.7.3)]
we have that (for $0<\Re(s)<1$)
\begin{equation}\zeta(s)\Gamma(s)=\int_{0}^{\infty}t^{s-1}\left(\frac{1}{e^t-1}-\frac{1}{t}\right)dt,\end{equation}
and for $x>0$
\begin{equation}\frac{1}{e^{2\pi t}-1}-\frac{1}{2\pi t}=\frac{1}{\pi}\int_{0}^{\infty}\sin(xt)\left(\frac{1}{e^t-1}-\frac{1}{t}\right)dt.\end{equation}
We may simply use (3.3) directly. (Alternatively, we may take the Mellin transform of (3.1) over the region $0<\Re(s)<1,$ invoke (3.2), and then find the desired $\Delta(x)$ using the functional equation (1.13).) Therefore, in conjunction with our solution in Theorem 2.1, the Fredholm alternative theorem tells us there are infinitely many solutions of (2.1) when $a=2\pi.$ 

\section{An associated integral} 
Here we offer an interesting integral involving the series in (2.17). 
\begin{theorem}For a positive real number $a,$ \begin{equation}\int_{0}^{\infty}\frac{\{t\}}{t}\left(\sum_{n\ge1}\frac{\mu(n)}{n}e^{-t/(na)}\right)dt=-\frac{1}{2}+\frac{1}{\pi}\int_{0}^{\infty}S(\frac{x}{2\pi a})e^{-x}dx.\end{equation}
\end{theorem}
\begin{proof} From Ivic [3, eq.(13)], we have for $0<c<1,$
\begin{equation}\int_{0}^{\infty}\frac{\{t\}}{t}e^{-t/a}dt=-\frac{1}{2\pi i}\int_{c-i\infty}^{c+i\infty}\frac{\zeta(s)}{s}a^s\Gamma(s)ds,\end{equation}
where $\{t\}$ denotes the fractional part of $t.$ We move the line of integration of the integral on the right side of (4.2) to the left $\Re(s)=b,$ $-1<b<0,$ and compute the residue from the pole of order two at $s=0.$ We use $\zeta(1)^{-1}=0,$ $\frac{d}{ds}\zeta(s)^{-1}|_{s=0}=-1,$ and $\Gamma(1)=-\gamma$ to get
\begin{equation}-\frac{1}{2\pi i}\int_{c-i\infty}^{c+i\infty}\frac{\zeta(s)}{s}a^s\Gamma(s)ds=\frac{1}{2}\log(a)-\frac{1}{2}\gamma+\frac{1}{2}\log(2\pi)-\frac{1}{2\pi i}\int_{b-i\infty}^{b+i\infty}\frac{\zeta(s)}{s}a^s\Gamma(s)ds.\end{equation}
Using uniform convergence of $\zeta(v)^{-1}=\sum_{n\ge1}\mu(n)/n^v$ for $\Re(v)\ge1,$ we have
\begin{equation}\int_{0}^{\infty}\frac{\{t\}}{t}\left(\sum_{n\ge1}\frac{\mu(n)}{n}e^{-t/(na)}\right)dt=-\frac{1}{2}-\frac{1}{2\pi i}\int_{b-i\infty}^{b+i\infty}\frac{\zeta(s)}{s\zeta(1-s)}a^s\Gamma(s)ds,\end{equation}
 for $-1<b<0.$ Using (1.13) we write the right side of (4.4) as 
\begin{equation}-\frac{1}{2\pi i}\int_{b-i\infty}^{b+i\infty}\frac{\zeta(s)}{s\zeta(1-s)}a^s\Gamma(s)ds=-\frac{1}{2\pi^2 i}\int_{b-i\infty}^{b+i\infty}\frac{\Gamma(1-s)\Gamma(s)\sin(\frac{\pi}{2}s)}{s}(2\pi a)^sds.\end{equation}
If $S(x)$ denotes the sine integral, then we have by Parseval's theorem
\begin{equation} \frac{1}{\pi}\int_{0}^{\infty}S(\frac{x}{2\pi a})e^{-x}x^{s-1}dx=-\frac{1}{2\pi^2 i}\int_{c_1-i\infty}^{c_1+i\infty}\frac{\Gamma(s-r)\Gamma(r)\sin(\frac{\pi}{2}r)}{r}(2\pi a)^rdr,\end{equation}
where $c_1$ is a real number confined to the strip of holomorphy when both $0>c_1>-1$ and $(s-c_1)>0.$ We put $s=1$ in (4.6) and note that we have set $-1<c_1<0.$ Comparing our resulting integrals gives the result.\end{proof}
If we instead work with the integral 
\begin{equation}-\frac{1}{2\pi i}\int_{c'-i\infty}^{c'+i\infty}\frac{\zeta(s)}{s}z^{s}ds=\{z\}-\frac{1}{2},\end{equation}
valid for $-1<c'<0,$ we may deduce a slightly different form of Theorem 4.1.
\begin{theorem} For a positive real number $a,$ \begin{equation}\int_{0}^{\infty}\frac{\{t\}-\frac{1}{2}}{t}\left(\sum_{n\ge1}\frac{\mu(n)}{n}(e^{-t/(na)}-1)\right)dt=\frac{1}{\pi}\int_{0}^{\infty}S(\frac{x}{2\pi a})e^{-x}dx.\end{equation}
\end{theorem}

We further note that it is known [2], and may be computed using the residue theorem from the integral on the right side of (4.4), that
\begin{equation}\int_{0}^{\infty}e^{-at}S(t)dt=\frac{1}{a}\tan^{-1}(\frac{1}{a}).\end{equation}

1390 Bumps River Rd. \\*
Centerville, MA
02632 \\*
USA \\*
E-mail: alexpatk@hotmail.com
\end{document}